\documentclass[12pt]{amsart}
\usepackage{amsmath, amssymb, amsthm, mathrsfs, lscape, graphicx, color, enumerate, textcomp, colonequals, fullpage}
\usepackage[backref]{hyperref}
\usepackage[alphabetic,backrefs,lite]{amsrefs}

\input xy
\xyoption{all}

\newtheorem{theorem}[subsection]{Theorem}

\newtheorem{lemma}[subsection]{Lemma}

\newtheorem{definition}[subsection]{Definition}

\newcommand{\Q}{\mathbb Q}

\renewcommand{\phi}{\varphi}

\newcommand{\Spec}{\mathrm{Spec}\ }

\DeclareMathOperator{\Gal}{Gal}
\DeclareMathOperator{\Res}{Res}

\title{A virtually ample field that is not ample}
\author{Padmavathi Srinivasan}
\date{\today}

\begin{document}

\begin{abstract}
 A field $K$ is called ample if for every geometrically integral $K$-variety $V$ with a smooth $K$-point, $V(K)$ is Zariski-dense in $V$. A field $K$ is virtually ample if some finite extension of $K$ is ample. We prove that there exists a virtually ample field that is not ample. 
\end{abstract}

\maketitle

\section{Introduction}
In the paper \cite{pop}, Pop identified a class of fields, which he called {\textup{large}} fields, over which a certain Galois theoretic problem (the regular solvability of finite split embedding problems) could be solved. Since their introduction, these fields have played an important role, not only in Galois theory, but also other fields such as arithmetic geometry, valuation theory and model theory. 
\begin{definition}({\cite[p.~2]{pop},\cite[Chapter~5]{jar}})
 A field $K$ is called {\textup{\textsf{ample}}} (or {\textup{\textsf{large}}} or {\textup{\textsf{anti-Mordellic}}}) if for every geometrically integral $K$-variety $V$ with a smooth $K$-point, $V(K)$ is Zariski-dense in $V$. 
\end{definition}
Pseudo-algebraically closed (PAC) fields and Henselian fields are ample; see \cite[Chapter 5]{jar} for more examples. See \cite{barysorfehm} for a survey of the theory of ample fields. 

\begin{definition}{\cite[p.~5]{barysorfehm}}
 A field $K$ is called {\textup{\textsf{virtually ample}}} if some finite extension of $K$ is ample.
\end{definition}

In \cite{barysorfehm}, Bary-Soroker and Fehm ask if every virtually ample field is ample. We prove that there exists a virtually ample field that is not ample (Theorem~\ref{uniformquadratic}). Our proof is modelled after the following result of Jarden and Poonen; however, there are difficulties in directly adapting their strategy, which we explain in our construction.

\begin{lemma}{\cite[Lemma~5.3]{jarpoon}}
 Every field $K$ admits a regular extension $K'$ such that
 \begin{enumerate}[(i)]
  \item Every smooth geometrically-integral $K$-curve $X$ has a point over an at most quadratic extension of $K'$ (possibly
depending on $X$).
  \item For every smooth geometrically integral $K$-curve $C$ of absolute genus $> 1$, we have $C(K') = C(K)$.
 \end{enumerate}
\end{lemma}

\section{Notation}
Let $K$ be a field. For an automorphism $\sigma \colon K \rightarrow K$, let $K^{\sigma}= \{x \in K \colon \sigma(x) = x \} $. A $K$-{\textsf{variety} is a separated scheme of finite type over $K$. A $K$-{\textsf{curve}} is a smooth, geometrically integral $K$-variety of dimension $1$. Let $\overline{K}$ be an algebraic closure of $K$. The {\textsf{absolute genus}} $g_C$ of a $K$-curve $C$ is the genus of a smooth projective model of $C_{\overline{K}}$. Let $\mathscr{C}_K$ denote the set of isomorphism classes of $K$-curves $C$ with $g_C \geq 2$ and let $\mathscr{V}_K$ denote the set of isomorphism classes of positive-dimensional geometrically integral $K$-varieties. We will abuse notation and often let $C \in \mathscr{C}_K$ and $V \in \mathscr{V}_K$ denote specific representatives of the corresponding isomorphism classes, when it is clear that the statement in question is independent of the particular representative chosen.

For a $K$-variety $V$ and an extension of fields $L \supset K$, let $V_L$ denote the base extension $V \times_K L$; similarly for a $K$-rational map $\phi$, let $\phi_L$ denote its base extension to $L$. Given a finite extension of fields $L \supset K$ and a quasi-projective $L$-variety $V$, let $\Res_{L/K} V$ denote the quasi-projective $K$-variety obtained by restriction of scalars.

Given an extension of fields $L \supset K$ and a geometrically integral $K$-variety $V$, let $L(V)$ denote the function field of the $L$-variety $V_L$. Given a field $L$ and a finite set $S = \{V_1,V_2,\ldots,V_n\}$ of geometrically integral varieties defined over subfields of $L$, let $L(S) \colonequals L(V_1 \times \cdots \times V_n).$
  
\section{Proof}

\begin{definition}
 An extension of fields $L \supset K$ is a {\textup{\textsf{good}}} extension if $C(L) = C(K)$ for every $C \in \mathscr{C}_K$. 
\end{definition}

\begin{lemma}{\label{goodness}}
 Let $L_0 \subset L_1 \subset L_2 \subset \ldots$ be a tower of finitely generated field extensions of $\Q$, and let $L = \varinjlim L_n$. Assume that $L_{n+1}$ is a good extension of $L_n$ for every $n$. Let $C \in \mathscr{C}_L$. Then $C(L)$ is finite. 
\end{lemma}
\begin{proof}
 Since $C$ is of finite type, $C$ is defined over $L_k$ for some $k$. Since $L_k$ is a finitely generated extension of $\Q$ and $g_C \geq 2$, the function field analogue of the Mordell conjecture \cite[Theorem 2.3]{lang} implies that $C(L_k)$ is finite. Since $L_{n+1}$ is a good extension of $L_n$ for every $n$, it follows that $C(L_k) = C(L_{k+1}) = \ldots = \varinjlim_{n \geq k} C(L_n) = C(L)$.
\end{proof}

\begin{lemma}{\label{factoring}}
 Let $K$ be a field of characteristic $0$, let $V,W \in \mathscr{V}_K$ and let $C \in \mathscr{C}_K$. Let $\psi \colon V \times W \dashrightarrow C$ be a rational map. Then $\psi$ factors through either the projection $\pi_1 \colon V \times W \rightarrow V$ or the projection $\pi_2 \colon V \times W \rightarrow W$.
\end{lemma}
\begin{proof}
 By replacing $V,W,C$ by $V_{\overline{K}},W_{\overline{K}},C_{\overline{K}}$ and $\psi$ by $\psi_{\overline{K}}$, we may assume $K = \overline{K}$ (if $\psi_{\overline{K}}$ factors through one of the projections, then so does $\psi$). We may view $\psi$ as an algebraic family of rational maps $W \dashrightarrow C$ parametrized by an open subvariety $U$ of $V$. Let $U' \subset U$ be the open subvariety that parametrizes non-constant rational maps $W \dashrightarrow C$. 
 
 If $U'$ is empty, then $\psi$ factors through $\pi_1$. Now assume that $U'$ is non-empty. Then $U'$ is irreducible, since $V$ is irreducible and $U'$ is open in $V$. Since $C(K(W)) \setminus C(K)$ is finite (by the de Franchis-Severi theorem \cite[p.~13]{lang}) and $U'$ is irreducible, the rational maps induced by elements of $U'$ are all equal, which implies that $\psi$ factors through $\pi_2$.
\end{proof}

\begin{lemma}{\label{nochange2}} 
Let $K$ be a field of characteristic $0$, let $V \in \mathscr{V}_K$ be a quasi-projective variety and let $L/K$ be a quadratic extension. Let $W := \Res_{L/K} V_L$. Then $K(W)$ is a good extension of $K$.
\end{lemma}
\begin{proof}
 Let $C \in \mathscr{C}_K$. Let $\sigma$ be the nontrivial element of $\Gal (L/K)$. Then,
 \begin{align*}\label{rat}\tag{$\star$}
 \begin{split}
 C(K(W)) &\cong \{ \textup{rational maps $\phi \colon W \dashrightarrow C$} \} \\ 
  &\cong \{ \textup{rational maps $\phi \colon W_L \dashrightarrow C_L$ such that $^\sigma \phi = \phi$} \}. 
 \end{split}
 \end{align*}
 Let $\rho \colon V_L \times V_L \rightarrow W_L$ be the isomorphism coming from the universal property of the restriction of scalars (since $V$ is defined over $K$, we get $^\sigma V_L \simeq V_L$). Let $s \colon V_L \times V_L \rightarrow V_L \times V_L$ be the involution that swaps the two factors. Then $^\sigma \rho = \rho s$. Given any rational map $\phi \colon W_L \dashrightarrow C_L$, let $\psi = \phi \rho$. Since $^\sigma \psi = {^\sigma}(\phi \rho) = (^\sigma \phi) (^\sigma \rho) = (^\sigma \phi) \rho s$, it follows that $^\sigma {\psi} = {\psi} s$ if and only if $^\sigma \phi = \phi$. Thus \ref{rat} can be rewritten as 
 \begin{align*}
 C(K(W)) &\cong \{ \textup{rational maps $\psi \colon V_L \times V_L \dashrightarrow C_L$ such that $^\sigma \psi = \psi s$} \}. 
 \end{align*}
 Lemma~\ref{factoring} tells us that any rational map $\psi \colon V_L \times V_L \dashrightarrow C_L$ factors through one of the projections to $V_L$. Let $\pi_1$ and $\pi_2$ be the two projection maps; these are defined over $K$. Without loss of generality, assume $\psi = \psi' \pi_1$ for some $\psi' \colon V_L \dashrightarrow C_L$. Also assume $^\sigma \psi = \psi s$. Then, $^\sigma \psi = \psi s = \psi' \pi_1 s = \psi' \pi_2$ and therefore by applying $\sigma$ once again to $^\sigma \psi$, we get $\psi = {^\sigma} (^\sigma \psi) = (^\sigma \psi') (^\sigma \pi_2) = (^\sigma \psi') \pi_2 $. Since $\psi$ factors through both projections, it is constant. Therefore $C(K(W)) = C(K)$. 
\end{proof}

\begin{definition}
 An {\textup{\textsf{involuted field}}} consists of a pair $(K,\iota)$ where $K$ is a field and $\iota \colon K \rightarrow K$ is an automorphism such that $\iota^2 = 1_K$ (Note that we allow $\iota = 1_K$). An {\textup{\textsf{extension of an involuted field}}} $(K,\iota)$ is an involuted field $(L,\tau)$ such that $L \supset K$ and $\tau$ extends the involution $\iota$ on $K$. A {\textup{\textsf{finitely generated involuted field}}} is an involuted field $(K,\iota)$ where $K$ is a finitely generated field extension of $\Q$.
\end{definition}

\begin{lemma}{\label{extension}}
Let $(K,\iota)$ be an involuted field. Let $F = K^{\iota}$. Given a quasi-projective $V \in \mathscr{V}_F$, there exists an extension $(M,\tau)$ of $(K,\iota)$ such that $M^{\tau}$ is a good extension of $F$, and $V(M) \neq V(K)$.
\end{lemma}
\begin{proof}
Let $M = K(V \times V)$. Let $s$ denote the involution on the $F$-variety $V \times V$ that swaps the two factors. Let $(s,\iota) \colon (V \times V) \times_{\Spec F} \Spec K \rightarrow (V \times V) \times_{\Spec F} \Spec K$ denote the product involution and let $\tau$ be the map on function fields induced by this morphism of $F$-varieties. Then $(M,\tau)$ is an extension of $(K,\iota)$.

Let $W = \Res_{K/F} V_K$. There is a natural isomorphism between $M^{\tau}$ and $F(W)$, since $W_K$ is naturally isomorphic to $V_K \times V_K$ and the $\Gal (K/F)$ action on $K(W)$ coincides with the involution $\tau$ on $M$ under this identification. The extension $F(W)/F$ is good by Lemma~\ref{nochange2}. Let $\pi \colon V \times V \rightarrow V$ be one of the projections. The restriction of $\pi$ to the generic point gives rise to an element of $V(M) \setminus V(K)$.
\end{proof}
  
\begin{theorem}{\label{uniformquadratic}}
 There exists a field $L$ of characteristic zero and a quadratic extension $M$ of $L$ such that
 \begin{itemize}
  \item $C(L)$ is finite for every $C \in \mathscr{C}_L$, and
  \item $V(M)$ is infinite for every $V \in \mathscr{V}_M$.
 \end{itemize}
 In particular, $L$ is non-ample, and $M$ is PAC and hence ample.
\end{theorem}
\begin{proof}
It suffices to prove the theorem with condition two replaced by the statement `$V(M)$ is infinite for every {\textit{quasi-projective}} $V \in \mathscr{V}_M$', since every $V \in \mathscr{V}_M$ contains a dense open quasi-projective (even affine) subvariety. So in the rest of the proof, we will implicitly assume that every $V \in \mathscr{V}_M$ is quasi-projective, so that we may freely use Lemmas~\ref{nochange2} and \ref{extension}.

We will first explain why the theorem does not simply follow from an iterated application of Lemma~\ref{extension}. Let $(M_0,\iota_0) = (\Q,1_{\Q})$. Let $V \in \mathscr{V}_{\Q}$. For every $k \geq 1$, let $(M_k,\iota_k)$ be the involuted field that we get by applying Lemma~\ref{extension} to the involuted field $(M_{k-1}, \iota_{k-1})$ and the variety $V_{M_{k-1}} \in \mathscr{V}_{M_{k-1}}$. Let $L_k = M_k^{\iota_k}$. Let $(M,\iota) = \varinjlim (M_k, \iota_k)$ and let $L = M^{\iota}$. By Lemmas~\ref{goodness} and \ref{extension} and the fact that every curve $C \in \mathscr{C}_L$ is defined over $L_k$ for some $k$, it follows that $C(L)$ is finite for every $C \in \mathscr{C}_L$. By construction $V(M)$ is infinite. Let $V' \in \mathscr{V}_{\Q}$ be another variety that we want to endow with infinitely many $M$-points by enlarging $M$. We can now repeatedly apply Lemma~\ref{extension} to $(M,\iota)$ and the variety $V'$ to obtain an involuted field $(O,\tilde{\iota})$ such that $V'(O)$ is infinite. Let $N = 
O^{\tilde{\iota}}$. The problem is that we can no longer appeal to Lemma~\ref{goodness} to say that $C(N)$ is finite for every $C \in \mathscr{C}_N$; even though $N$ is obtained from $L$ by a successive series of good extensions, the field $L$ itself is not a {\textit{finitely generated}} extension of $\Q$. Another issue is that as we enlarge our field $M_0$ to give all varieties in $\mathscr{V}_\Q$ infinitely many points, we also inadvertently end up enlarging $M$, which gives us {\textit{new}} elements of $\mathscr{V}_{M}$.  We now need to ensure that every element of $\mathscr{V}_M$ has infinitely many points over $M$, not just the varieties that come from $\mathscr{V}_{\Q}$! 

We resolve both issues by inductively constructing 
  \begin{itemize}
   \item a tower of {\textit{finitely generated}} involuted fields $(\Q, 1_{\Q}) = (M_0, \iota_0) \subset (M_1, \iota_1) \subset (M_2, \iota_2) \subset \ldots$ such that each successive extension in the corresponding tower of fixed fields $L_0 \subset L_1 \subset L_2 \subset \ldots$ is a good extension, and, 
   \item an enumeration of $\mathscr{V}_{L_k}$ for every $k$.
  \end{itemize}
During each inductive step, we enlarge $(M_k,\iota_k)$ in a way that gives new points to some finite subset of the varieties enumerated thus far, in such a way that in the limit, all the varieties in $\mathscr{V}_L$ acquire infinitely many points over $M$. We now make the inductive step more precise.

To begin the induction, set $(M_0,\iota_0) = (\Q,1_{\Q})$ and fix any enumeration of $\mathscr{V}_{\Q}$. Now assume that we have defined $(M_k,\iota_k)$ for all $k \leq n$, and that we have fixed an enumeration $\{ V_{k,0}, V_{k,1}, V_{k,2}, \ldots \}$ of $\mathscr{V}_k$ for every $k \leq n$. Let $S = \{ V_{k,j} \times V_{k,j} \ |\ k = n \ \textup{or}\ j = n \}$ . Let $M_{n+1} = M_n(S)$. A repeated application of Lemma~\ref{extension} lets us extend $\iota_n$ to an involution $\iota_{n+1}$ of $M_{n+1}$ such that the corresponding extension of fixed fields $L_{n+1}/L_n$ is a good extension. The field $M_{n+1}$ is a finitely generated extension of $\Q$, since $M_n$ is a finitely generated extension of $\Q$, and function fields of geometrically integral varieties over finitely generated fields are finitely generated. Let $\{ V_{n+1,0}, V_{n+1,1}, \ldots \}$ be any enumeration of the set of geometrically integral $L_{n+1}$-curves. This completes the induction.

Lemma~\ref{goodness} tells us that $C(L)$ is finite for every $C \in \mathscr{C}_L$. Fix $V \in \mathscr{V}_L$. We have to show that $V(M)$ is infinite. Suppose that this is false. Since $M = \varinjlim M_i$, there exists a positive integer $k$ such that $V$ is defined over $L_k$ and $V(M) = V(M_k)$. Let $p = k+1$. Then there exists an integer $q$ such that $V_{M_p} = V_{p,q}$. Let $r = \mathrm{max}(p,q)+1$. By the construction of $M$, it follows that there exists an element of $V(M_r) \setminus V(M_k)$. This is a contradiction, since $V(M_r) \subset V(M) = V(M_k)$.

So far we have proved that $V(M)$ is infinite for every quasi-projective $V \in \mathscr{V}_L$, but we need to prove that this holds for every quasi-projective $V \in \mathscr{V}_M$. Let $V \in \mathscr{V}_M$ and let $V^\iota \in \mathscr{V}_M$ be the conjugate variety. Then, there exists a quasi-projective $W \in \mathscr{V}_L$ such that $W_M = V \times V^{\iota}$. By applying the theorem to $W \in \mathscr{V}_L$ and using the identification $V(M) \times V^{\iota}(M) = W(M)$, we conclude that either $V(M)$ or $V^\iota(M)$ is infinite. Since $\iota$ induces a bijection between $V(M)$ and $V^{\iota}(M)$, we conclude that $V(M)$ is infinite.
\end{proof}

\subsection*{Remarks.}
 \begin{itemize}
  \item Since each step in the tower $M_0 \subset M_1 \subset M_2 \subset \cdots$ is a finitely generated regular extension, it follows that $M = \varinjlim M_i$ and its subfield $L$ are both Hilbertian. It is not known whether there are infinite non-ample fields that are not Hilbertian.
  \item The technique of constructing towers of field extensions by adjoining generic points of absolutely irreducible varieties goes back to \cite[Theorem~2.6]{geyerjar} where it was used to construct a non PAC field with separably closed Henselian closure.
  \item One can also produce a virtually ample non-ample field in characteristic $p$ by starting with $M_0 = \mathbb{F}_p(t)$ instead of $\mathbb{Q}$ in the theorem above, using the fact that $C(M_0)$ is finite for any non-isotrivial $M_0$-curve $C$ of absolute genus $>1$. 
 \end{itemize}

\section*{Acknowledgements}
I would like to thank my advisor, Bjorn Poonen, for suggesting this problem to me, for several helpful discussions, and for his suggestions for improving the exposition. I also thank Moshe Jarden for comments.

\end{document}